\documentclass[12pt, reqno]{amsart}
\usepackage{amsmath}
\usepackage{amssymb}
\usepackage{epsfig}
\usepackage{graphicx}
\usepackage{color}
\usepackage{fullpage}
\definecolor{shadecolor}{gray}{0.875}
\usepackage{amscd}

\numberwithin{equation}{section}

\input xy
\xyoption{all}

\calclayout
\allowdisplaybreaks[3]

\theoremstyle{plain}
\newtheorem{prop}{Proposition}[section]

\newtheorem{theo}[prop]{Theorem}
\newtheorem{coro}[prop]{Corollary}

\newtheorem{lemm}[prop]{Lemma}

\theoremstyle{definition}
\newtheorem{defi}[prop]{Definition}

\theoremstyle{remark}

\def\lra{\longrightarrow}
\def\ra{\rightarrow}

\def\cC{{\mathcal C}}

\def\cH{{\mathcal H}}

\def\cN{{\mathcal N}}
\def\cO{{\mathcal O}}
\def\cR{{\mathcal R}}

\def\cU{{\mathcal U}}

\def\bP{{\mathbb P}}
\def\bQ{{\mathbb Q}}
\def\bR{{\mathbb R}}
\def\bZ{{\mathbb Z}}

\def\dim{\mathrm{dim}}

\def\Eff{\overline{\mathrm{Eff}}}
\def\Pic{\mathrm{Pic}}

\def\Mor{\mathrm{Mor}}
\def\Nef{\mathrm{Nef}}

\def\Pic{\mathrm{Pic}}
\def\Rat{\overline{\mathrm{Rat}}}
\def\wt{\widetilde}
\def\AJ{\mathrm{AJ}}

\makeatother
\makeatletter

\author{Nobuki Shimizu}
\address{Department of science, Graduate School of Science and Technology, Kumamoto University, 2-39-1 Kurokami, Chuo-ku, Kumamoto 860-8555, Japan}
\email{nshimizu5023@gmail.com}

\author{Sho Tanimoto}
\address{Graduate School of Mathematics, Nagoya University, Furocho Chikusa-ku, Nagoya, 464-8602, Japan}
\email{sho.tanimoto@math.nagoya-u.ac.jp}

\subjclass[2020]{Primary : 14H10. Secondary : 14J45.}

\title{The spaces of rational curves \\ on del Pezzo threefolds of degree one}

\begin{document}
\date{\today}

\begin{abstract}
We prove the irreducibility of moduli spaces of rational curves on a general del Pezzo threefold of Picard rank $1$ and degree $1$. As corollaries, we confirm Geometric Manin's conjecture and enumerativity of certain Gromov-Witten invariants for these threefolds.
\end{abstract}

\maketitle

\section{Introduction}

Rational curves on Fano varieties are extensively studied due to their prominent role in classification theory. Indeed the study of low degree rational curves on Fano varieties is a classical subject, and lines, conics, and twisted cubics are well-studied by many algebraic geometers. These date back to at least two centuries ago represented by a result of $27$ lines on a cubic surface. See, e.g., a survey paper \cite{CZ20} and references therein for a history and results on low degree rational curves.

Mori proved in \cite{Mori82} that for any smooth Fano variety $X$ and any point $x \in X$, there exists a rational curve on $X$ passing through $x$ using his famous Bend and Break technique. This idea has been further developed in \cite{Cam92} and \cite{KMM92} proving that any smooth Fano variety is rationally connected, i.e., there exists a family $\pi : \mathcal C \to M$ of rational curves on $X$ such that for any two general points, there is a member of $\pi$ passing through these points. Thus Fano varieties possess a lot of rational curves, and it is natural to study the space of rational curves on a Fano variety. In particular, one may ask what the dimension and the number of irreducible components of the moduli space of rational curves are.

As mentioned above, these questions have been well-studied for low degree rational curves, and one may ask the same questions for higher degree rational curves. One of the pioneering works in this direction is \cite{HRS04} which studied the irreducibility of the moduli space of rational curves of degree $d$ on a low degree general hypersurface using an inductive proof on $d$ based on Bend and Break. This method has been further developed and generalized in \cite{BK13} and \cite{RY16}, and we have a fairly complete understanding of the moduli space of rational curves on a general Fano hypersurface. See also \cite{BP17} for another approach to this problem using the circle method, an important technique from analytic number theory.

In this paper, we consider another class of Fano varieties, i.e., smooth Fano threefolds and the space of rational curves on them. In \cite{LT19}, Lehmann and the second author proposed an approach to understand moduli spaces of rational curves using the perspective of a version of Manin's conjecture which has been developed in a series of papers \cite{FMT}, \cite{BM}, \cite{Peyre}, \cite{BT}, \cite{Peyre03}, and \cite{LST18}, and this approach has been further developed and generalized in \cite{LT21}, \cite{LT19b}, \cite{BLRT}, and \cite{LTSectionII} for smooth projective threefolds. In \cite{BLRT}, Beheshti, Lehmann, Riedl, and the second author established two main results to classify rational curves on smooth Fano threefolds. The first result is the Movable Bend and Break Lemma (Theorem~\ref{mbb}) which claims that a free rational curve of high enough anticanonical degree degenerates to the union of two free curves in the moduli space of stable maps. The second result is a classification of $a$-covers for smooth Fano threefolds and its consequences for moduli spaces of rational curves (Theorem~\ref{2para}). These two results reduce the problem of irreducibility of moduli spaces of rational curves on a given Fano threefold to a finite computation, i.e., checking irreducibility of moduli spaces of low degree rational curves. In this paper, we employ this strategy and apply it to del Pezzo threefolds of Picard rank $1$ and degree $1$, proving irreducibility of the moduli spaces parametrizing rational curves of degree $d$.

A del Pezzo threefold of Picard rank $1$ is a smooth Fano threefold $X$ with $\mathrm{Pic}(X) = \mathbb ZH$ and $-K_X = 2H$. Such threefolds have been classified by Fano and Iskovskih (\cite{Isk77}, \cite{Isk78}, and \cite{Isk79}) and the degree $H^3$ can take any integer value from $1$ to $5$, each corresponding to one family of Fano threefolds. Starr proved the irreducibility of the moduli space of rational curves of degree $d$ in \cite{Starr00} when $H^3 =3$, i.e., $X$ is a smooth cubic threefold in $\mathbb P^4$. In \cite{Cas04}, Castravet settled this issue when $H^3 = 4$, i.e., $X$ is a smooth complete intersection of two quadrics in $\mathbb P^5$. In \cite{LT19}, Lehmann and the second author produced a uniform treatment for del Pezzo threefolds with $H^3 \geq 2$ using the perspective of Manin's conjecture. In this paper we consider the last remaining case, i.e., smooth del Pezzo threefolds of Picard rank $1$ and degree $1$. Here is our main theorem:

\begin{theo}
\label{theo:main}
Let $X$ be a smooth Fano threefold defined over an algebraically closed field $k$ of characteristic $0$ such that $\Pic(X)=\bZ H$, $-K_X=2H$ and $H^3=1$. Assume that $X$ is general in its moduli. Let  $\overline M_{0, 0}(X, d)$ be the moduli space of stable maps of $H$-degree $d$. Assume that $d \geq 2$. Then $\overline M_{0, 0}(X, d)$ consists of two irreducible components:$$\overline M_{0, 0}(X, d)=\cR_d\cup\cN_d$$ such that a general element $(C, f)\in\cR_d$ is a birational stable map from an irreducible curve to the image and any element $(C, f)\in\cN_d$ is a stable map of degree $d$ to a line in $X$.
\end{theo}

Our proof is based on induction on $d$. The inductive step is completed by combining \cite{LT19} and \cite{BLRT}, thus our task here is to prove the base cases, i.e., when $d = 1$ and $d = 2$. When $d = 1$, $\overline M_{0, 0}(X, 1)$ is irreducible by work of Tikhomirov (\cite{T}). Hence our efforts will be focused on the irreducibility of the space of $H$-conics mapping birationally to the image. This will be done by establishing Movable Bend and Break for $H$-conics which is outside of the degree range of Movable Bend and Break proved in \cite{BLRT}.

As corollaries we confirm Geometric Manin's conjecture and enumerativity of certain Gromov-Witten invariants for general del Pezzo threefolds of Picard rank $1$ and degree $1$. See Section~\ref{sec:applications} for more details.

\

Here is a road map of this paper: in Section~\ref{sec:prelim}, we recall basic definitions and previous results which are important for our study. In Section~\ref{sec:H-conics}, we prove the irreducibility of the space of $H$-conics on a general del Pezzo threefold of Picard rank $1$ and degree $1$ by establishing Movable Bend and Break for free $H$-conics. In Section~\ref{sec:higherdegree}, we prove our main theorem, Theorem~\ref{theo:main} by combining \cite{LT19} and \cite{BLRT}. In Section~\ref{sec:applications}, we discuss applications of our work to Geometric Manin's conjecture and enumerativity of certain Gromov-Witten invariants for general del Pezzo threefolds of Picard rank $1$ and degree $1$.

\

\noindent
{\bf Notation}: Let $k$ be an algebraically closed field of characteristic $0$. A variety over $k$ is an integral separated scheme of finite type over $k$. A component of a scheme $X$ means an irreducible component of $X$ unless stated otherwise.

Let $X$ be a projective variety of Picard rank $1$ defined over $k$ with an ample generator $H$ for $\mathrm{Pic}(X)$. Then $\overline{M}_{0,n}(X, d)$ is the coarse moduli space of stable maps of $H$-degree $d$ with $n$ marked points.

For a smooth projective variety $X$, let $N^{1}(X)_{\bZ}$ denote the quotient of the group of the Cartier divisors by numerical equivalence and let $N_{1}(X)_{\bZ}$ denote the quotient of the group of integral $1$-cycles by numerical equivalence. We set $N^{1}(X)=N^{1}(X)_{\bZ}\otimes_{\bZ}\bR$ and $N_{1}(X)=N_{1}(X)_{\bZ}\otimes_{\bZ}\bR$ which are finite dimensional real vector spaces. Let $\Eff^{1}(X)$ and $\Eff_{1}(X)$ denote the pseudo-effective cones of divisors and curves, and the intersections of these cones with $N^{1}(X)_{\bZ}$ and $N_1(X)_{\bZ}$ are denoted by $\Eff^{1}(X)_{\bZ}$ and $\Eff_{1}(X)_{\bZ}$ respectively. We also denote the nef cone of divisors and nef cone of curves by $\Nef^{1}(X)$ and $\Nef_{1}(X)$. Moreover the intersections of these cones with $N^{1}(X)_{\bZ}$ and $N_1(X)_{\bZ}$ are denoted by $\Nef^{1}(X)_{\bZ}$ and $\Nef_{1}(X)_{\bZ}$ respectively.

\bigskip

\noindent
{\bf Acknowledgements:}
This work is based on the first author's Master thesis \cite{Shi21} at Kumamoto University. The authors would like to thank Brian Lehmann for helpful discussions and comments on an early draft of this paper. The authors would also like to thank Lars Halvard Halle for answering our question regarding Kulikov models. The authors would like to thank the anonymous referees for careful reading of the paper and helpful comments which significantly improved the exposition of the paper.

The second author was partially supported by Inamori Foundation, by JSPS KAKENHI Early-Career Scientists Grant number 19K14512, by JSPS Bilateral Joint Research Projects Grant number JPJSBP120219935, and by MEXT Japan, Leading Initiative for Excellent Young Researchers (LEADER).

\section{Preliminaries}
\label{sec:prelim}

In this paper we work over an algebraically closed field $k$ of characteristic 0.
We recall some definitions and previous results in this section.

First we introduce certain birational invariants which play crucial roles in our study:

\begin{defi}[{\cite[Definition 2.2]{HTT15}}]
Let $X$ be a smooth projective variety and let $L$ be a big and nef $\bQ$-divisor on $X$. The {\it Fujita invariant} is $$a(X, L):=\min\{t\in\bR\mid t[L]+[K_{X}]\in\Eff^{1}(X)\}.$$ If $L$ is not big, we set $a(X, L)=\infty$.
\end{defi}
When $X$ is a singular projective variety, we define the $a$-invariant by pulling  back to a smooth resolution $\beta:\widetilde{X}\ra X$: $$a(X, L):=a(\widetilde{X}, \beta^{*}L).$$ This is well-defined by \cite[Proposition 2.7]{HTT15}. When $L$ is big, it follows from \cite{BDPP13} that $a(X, L)$ is positive if and only if $X$ is uniruled, i.e., there exist a subvariety $Y$ and a generically finite dominant rational map $\bP^{1}\times Y\dashrightarrow X$.


\begin{defi}[{\cite[Definition 2.8]{HTT15}}]
Let $X$ be a uniruled smooth projective variety and $L$ be a big and nef $\bQ$-divisor on $X$. We define $b(X, L)$ to be the codimension of the minimal supported face of $\Eff^{1}(X)$ containing the class $a(X, L)[L]+[K_{X}]$.
\end{defi}

When $X$ is singular, we define the $b$-invariant by pulling back to a smooth resolution $\beta:\widetilde{X}\ra X$: $$b(X, L)=b(\widetilde{X}, \beta^{*}L).$$ It is well-defined because of birational invariance \cite[Proposition 2.10]{HTT15}.


These invariants play central roles in Manin's conjecture as these invariants appear as the exponents of the asymptotic formula for the counting function of rational points on a smooth projective rationally connected variety after removing the contribution from an exceptional set. The following theorem can be used to describe the exceptional set for Manin's conjecture:

\begin{theo}[{\cite[Theorem 1.1]{HJ17} and \cite[Theorem 3.3]{LT19}}]
\label{proper closed}
Let $X$ be a smooth uniruled projective variety and $L$ be a big and nef $\bQ$-divisor on $X$. Let $V$ be the union of all subvarieties $Y$ such that $a(Y, L|_{Y})>a(X, L)$. Then $V$ is a proper closed subset of $X$ and its components are precisely the maximal elements in the set of subvarieties with higher $a$-invariant.
\end{theo}

The closed set $V$ is explicitly understood for smooth Fano threefolds in \cite{BLRT}:

\begin{theo}[{\cite[Theorem 4.1]{BLRT}}]
\label{theo:higher_a-inv}
Let $X$ be a smooth Fano threefold of Picard rank 1 and $Y$ be a subvariety of dimension 2 of $X$ such that $a(Y, -K_{X})>a(X, -K_{X})$. Then $Y$ is swept out by $-K_{X}$-lines.
\end{theo}

Let $\Mor(\bP^{1}, X)$ be the quasi-projective scheme parametrizing maps $\bP^{1}\ra X$. This is constructed in \cite{Gro95}.

\begin{defi}
Let $X$ be a smooth projective variety and let $\alpha\in\mathrm{Eff}_{1}(X)_{\bZ}$. The set of components of $\Mor(\bP^{1}, X)$ parametrizing curves of class $\alpha$ is denoted by $\Mor(\bP^{1}, X, \alpha)$. For an open subset $U\subset X$, $\Mor_{U}(\bP^{1}, X, \alpha)$ denotes the sublocus of $\Mor(\bP^{1}, X, \alpha)$ parametrizing curves which meet $U$.
\end{defi}


Here is an important connection between Manin's conjecture and properties of moduli spaces of rational curves:
\begin{theo}[{\cite[Theorem 4.6]{LT19}}]
Let $X$ be a smooth weak Fano variety, i.e., $X$ is projective and $-K_X$ is big and nef. Let $V$ be the union of the subvarieties $Y$ of $X$ with $a(Y, -K_{X}|_{Y})>a(X, -K_{X})$. This is proper closed by Theorem \ref{proper closed}. Let $U$ be the complement of $V$. Then any component $M$ of $\Mor_{U}(\bP^{1}, X, \alpha)$ is a dominant component, i.e, for the universal family $\pi : \mathcal C \to M$, the evaluation map $s : \mathcal C \to X$ is dominant. Hence it satisfies $$\dim\,M=-K_{X}\cdot\alpha+\dim\,X$$ for any $\alpha\in\Nef_{1}(X)_{\bZ}$.
\end{theo}

We also consider the moduli space of stable maps. Let $X$ be a smooth projective variety and $\beta$ be an element in $N_{1}(X)_{\bZ}$. Let $\overline{M}_{0, n}(X, \beta)$ be the coarse moduli space of stable maps of genus 0 and class $\beta$ with $n$ marked points. (See \cite{BM96} for the definitions and basic properties of this moduli space.) For a smooth Fano threefold we let $\Rat(X)$ denote the union of the components of $\overline M_{0, 0}(X)$ that generically parametrize stable maps with irreducible domains.


One of the main theorems in \cite{BLRT} is Movable Bend and Break Lemma which is a key to a solution of Batyrev's conjecture for smooth Fano threefolds in \cite{BLRT}. Here we state Movable Bend and Break lemma for smooth Fano threefolds of Picard rank $1$:

\begin{theo}[{\cite[Theorem 6.7]{BLRT}} Movable Bend-and-Break Lemma] \label{mbb}
Let $X$ be a smooth Fano threefold of Picard rank $1$. Let $M$ be a component of $\Rat(X)$ that generically parametrizes free curves. Suppose that a general curve $C$ parametrized by $M$ has anticanonical degree $\geq 5$. Then $M$ contains a stable map of the form $f:Z\ra X$ where $Z$ has two components and the restriction of $f$ to each component realizes this component as a free curve on $X$.
\end{theo}


The next main theorem from \cite{BLRT} follows from a classification of $a$-covers for smooth Fano threefolds. Such a classification was first obtained in \cite{LT17} for Fano threefolds of Picard rank $1$:
\begin{theo}[{\cite[Theorem 1.3]{BLRT}}]\label{2para}
Let $X$ be a smooth Fano threefold. Let $M$ be a component of $\Rat(X)$ and let $\cC\ra M$ be the corresponding component of $\overline M_{0,1}(X)$. Suppose that the evaluation map $\mathrm{ev} : \cC\ra X$ is dominant and its general fibers are not irreducible. Then either:
\begin{itemize}
\item$M$ parametrizes a family of stable maps whose images are $-K_{X}$-conics, or
\item$M$ parametrizes a family of curves contracted by a del Pezzo fibration $\pi:X\ra Z$.
\end{itemize}
\end{theo}


Finally we will need the following result regarding the variety of $H$-lines for a general del Pezzo threefold of Picard rank $1$ and degree $1$:
\begin{theo}[{\cite{T}}]\label{var of lines}
Assume that $X$ is a smooth del Pezzo threefold of Picard rank $1$ and degree $1$, i.e., $X$ is a smooth Fano threefold such that $\Pic(X)=\bZ H$, $-K_{X}=2H$, and $H^{3}=1$. Let $\mathcal R_1=\overline M_{0, 0}(X, 1)$. Suppose $X$ is general in its moduli. Then $\mathcal R_1$ is irreducible and smooth. Furthermore the Abel-Jacobi mapping from $\mathcal R_1$ to the intermediate Jacobian $\mathrm{IJ}(X)$ of $X$ $$\AJ:\mathcal{R}_1 \ra\mathrm{IJ}(X)$$ is generically finite to the image.
\end{theo}

For definitions of intermediate Jacobians and the Abel-Jacobi maps, see \cite[Section 12]{Voisin}.

\section{$H$-conics on del Pezzo threefolds of degree $1$}
\label{sec:H-conics}

Let $X$ be a del Pezzo threefold of Picard rank $1$ and degree $1$, i.e., $X$ is a smooth Fano threefold such that $\Pic(X)=\bZ H$, $-K_{X}=2H$, and $H^{3}=1$. In this paper we assume that $X$ is general in its moduli. We denote the morphism of $X$ associated to the complete linear system of $-K_{X}$ by $\Phi_{|-K_{X}|} : X \lra W\subset\bP^{6}$. This is a degree $2$ finite morphism to the Veronese cone $W$ in $\bP^{6}$ ramified along the intersection of $W$ and a cubic hypersurface avoiding the cone point. We denote the involution associated to the double cover $\Phi_{|-K_{X}|}:X\ra W$ by $\iota : X \to X$.
Let $S$ be the pullback of a hyperplane on $\mathbb P^6$ so that $S\sim2H$, i.e., $S\in|-K_{X}|=\bP^{6}$.

First let us describe surfaces $Y$ of $X$ whose $a$-invariants are same as the $a$-invariant of $X$:

\begin{lemm}\label{iitaka one}
Let $Y$ be a surface on $X$ and let $\beta:\wt Y\ra Y$ be the resolution. If $a(Y, H|_Y) = a(X, H)=2$, then the Iitaka dimension $\kappa(2\beta^*H+K_{\wt Y})$ is $1$.
\end{lemm}

\begin{proof}
Assume that $\kappa(2\beta^*H+K_{\wt Y})=0$. By \cite{Hor10} (see also \cite[Theorem 8.10]{BLRT21}), $(\wt Y, \beta^*H)$ is birationally equivalent to a quadric surface $(Q, \cO(1))$ as a polarized surface. The resolution $\beta$ factors as $\wt Y\ra Y'\ra Y$ where $Y'\ra Y$ is the normalization. We denote it $\beta':Y'\ra Y$. Since $\beta'^*H$ is ample, we conclude that $$(Y', \beta'^*H)\cong(Q, \cO(1)).$$ Since $Y$ is a divisor in a smooth variety, $K_Y$ still makes sense. Hence we must have $$2H+K_Y\sim 0.$$ 
Indeed, since $(Y', \beta'^*H)$ is isomorphic to $(Q, \cO(1))$, we have $2\beta'^*H + K_{Y'} \sim 0$.
By taking the pushforward, our assertion follows.
Let $e\in\bZ_{>0}$ be a positive integer such that $Y\sim eH$. Thus by adjunction we have $K_Y\sim(e-2)H$. But these relations imply that $eH\sim0$, a contradiction. 
\end{proof}

Next we will prove that a general rational curve of $H$-degree $\geq 2$ is very free following arguments in \cite[Lemma 8.1]{LT21}:

\begin{prop}[{\cite[Lemma 8.1]{LT21}}]
\label{prop:veryfree}
Let $X$ be a smooth del Pezzo threefold of Picard rank $1$ and degree $1$. Assume that $X$ is general in its moduli. Let $M\subset\overline M_{0, 0}(X, d)$ be a dominant component generically parametrizing birational stable maps. When $d\geq2$, a general member $C\in M$ is very free. 
\end{prop}

\begin{proof}
The proof below is taken from \cite[Lemma 8.1]{LT21}. We include it for completeness of the paper. It is enough to show that the evaluation map $$\mathrm{ev}_{2}:M^{(2)}\ra X\times X$$ is dominant where $M^{(2)}$ is the component of $\overline M_{0, 2}(X, d)$ above $M$. Suppose that it is not dominant. Since there is at least a one-parameter family of curves through a general point, the image must be an irreducible divisor $D$ in $X\times X$.

From the assumption we have $-K_X\cdot C = 2H\cdot C=2d\geq4$. For the fiber of the evaluation map $\mathrm{ev}:M\ra X$ at a general point $x_{1}\in X$ , we have $$\dim(\mathrm{ev}^{-1}(x_{1}))=2d-2\geq2.$$ Take a component $N\subset\mathrm{ev}^{-1}(x_{1})$ and let $S$ be the surface swept out by $N$. Then there exists a smooth resolution $\beta:\widetilde S\ra S$ such that $\beta^{-1}(x_{1})$ is divisorial in $\wt S$. Let $\widetilde C$ be the strict transform of $C$. Then we must have 
$$-K_{\widetilde{S}}\cdot\widetilde{C}-1=2d-2,$$
hence we have 
$$-K_{\widetilde{S}}\cdot\widetilde{C}=2d-1.$$ 
Since $a(S, 2H)2\beta^{*}H+K_{\wt S}$ is pseudo-effective, it follows that $$a(S, 2H)2\beta^{*}H\cdot\wt C+K_{\wt S}\cdot\wt C\geq0.$$ Thus we obtain $$a(S, 2H)\geq\frac{2d-1}{2d}.$$

Now we can write $a(S, 2H)=\frac{2}{n}$ or $\frac3n$ by \cite[Lemma 4.5]{LT21} where $n$ is a positive integer. When $d\geq3$, we must have $a(S, 2H)=1$. When $d=2$, we have $a(S, 2H)=1$ or $\frac{3}{4}$. We will argue for each case. 

When $a(S, 2H)=1$, we have $\kappa(2\beta^{*}H+K_{\wt S})=1$ by Lemma \ref{iitaka one}. In this case the canonical map $\pi:\wt S\ra B$ associated to $2\beta^{*}H+K_{\wt S}$ exists and we have the Zariski decomposition $$2\beta^{*}H+K_{\wt S}=eF+E$$ with a general fiber $F\subset\wt S$ of $\pi$, an effective divisor $E$ and a positive integer $e$. Note that $F$ satisfies $\beta^*H\cdot F = 1$ so its image on $X$ is an $H$-line. Then $$2\beta^{*}H\cdot\wt C+K_{\wt S}\cdot\wt C=2d-(2d-1)=1$$ implies that $(eF+E)\cdot\wt C>0.$ This means that $\wt C$ maps to $B$ dominantly and this implies that $B$ is rational. Thus we conclude that the variety of $H$-lines $\mathcal R_1$ is covered by rational curves $B$ as $x_1$ varies. Now let us consider the Abel-Jacobi mapping $$\AJ:\mathcal R_1 \ra\mathrm{IJ}(X)$$ where IJ$(X)$ is the intermediate Jacobian of $X$. By Theorem \ref{var of lines} $\AJ$ is generically finite to the image. Since $\mathrm{IJ}(X)$ has no rational curves, all such curves contained in $\mathcal R_1$ are contracted by this mapping. Thus there are only finitely many rational curves on $\mathcal R_1$ which contradicts with the fact that $B$ varies as $x_1$ varies.

When $a(S, 2H)=\frac{3}{4}$, let $S'$ be the normalization of $S$. From \cite[Lemma 2.4 and Proposition 2.5]{S}, $S'$ is smooth along the image of $\wt C$ which is the strict transform of $C$. Furthermore take the minimal resolution $\beta : \wt S' \to S$. Then there is a birational map $f:\wt S'\ra\bP^2$ by \cite[Theorem 5.5]{LT21} and it maps the 2-dimensional family of $\wt C'\subset \wt S'$ which is the strict transform of $C$ to the lines passing through one point on $\mathbb P^2$. Indeed, it follows from  \cite[Theorem 5.5]{LT21} that we have $2\beta^*H = f^*\mathcal O(4) - K_{\wt S'/\mathbb P^2}$ so that $f_*\wt C'\cdot\mathcal O(1) = 1$. Note that we have $d = 2$ in our situation. Then such lines admit only 1-dimensional family on $\bP^2$, a contradiction.
\end{proof}


Next we start to analyze $H$-conics. First we show that a general $H$-conic maps birationally to the image under $\Phi_{|-K_X|}$:

\begin{lemm}\label{quartic}
Let $M\subset\overline{M}_{0,0}(X,2)$ be a dominant component generically parametrizing birational stable maps. Let $C\in M$ be a general member. Then $\Phi_{|-K_{X}|}|_{C}:C\ra C'$ is birational and $C'$ is a quartic rational curve in $\bP^{6}$ where $C'=\Phi_{|-K_{X}|}(C)$.
\end{lemm}

\begin{proof}
Let $C$ be a general very free rational curve on $X$ such that $H.C=2$. Then there are three possibilities for the image $C'$ of $C$ via $\Phi_{|-K_{X}|}$:
\begin{itemize}
\item$C'$ is a line and $\Phi_{|-K_{X}|}|_C: C\ra C'$ has degree 4;
\item$C'$ is a conic and $\Phi_{|-K_{X}|}|_C: C\ra C'$ has degree 2;
\item$C'$ is a quartic rational curve and $\Phi_{|-K_{X}|}|_C: C\ra C'$ is birational.
\end{itemize}

We will consider these situations. For the first case every line on $W$ passes through the singular point, but this contradicts with the assumption that $C$ is general so that $C$ avoids any codimension $2$ locus. 

For the second case note that since $C$ and $C'$ are smooth by \cite[3.14 Theorem]{Kollar} , we have the Hurwitz formula $$2g(C)-2=2(2g(C')-2)+2.$$ Let $Q$ be a cubic hypersurface in $\bP^6$ such that $\Phi_{|-K_{X}|}:X\ra W$ is ramified along $Q\cap W$. Since we have $S'\cdot C'=2$ where $S'$ is a hyperplane class in $\mathbb P^6$, we have $Q\cdot C'=6$ so that the degree of the ramification divisor of $\Phi_{|-K_{X}|}|_C: C\ra C'$ is less than or equal to 6. Since we have $-K_{W}\sim\frac{5}{2}S'$, we conclude $-K_{W}\cdot C'=5$. But then among points in $C'\cap Q$ there must be two points with multiplicity 2 and the dimension decreases by 1 for having a multiplicity 2 point, so that one can show that the dimension is $5-1-1=3<4$  and this contradicts with $-K_{X}\cdot C=4$. Thus we conclude that for a general member $C$ we only have the third case.
\end{proof}


Next we show that for a general very free $H$-conic $C$, its image via $\Phi_{|-K_X|}$ spans $\mathbb P^4$ so that it is a normal quartic rational curve:
\begin{lemm}
Let $M\subset\overline{M}_{0,0}(X,2)$ be a dominant component generically parametrizing birational stable maps. Let $C\in M$ be a general member and $C'=\Phi_{|-K_{X}|}(C)$.
Then $C'$ spans a linear subspace $\bP^{4}\subset\bP^{6}$.
\end{lemm}

\begin{proof}
Let $C'\subset W$ be a general quartic rational curve avoiding cone point. Then we have $-K_{W}\cdot C' = 10$ so that the parameter space of normal quartic rational curves on $W$ has dimension 10. Assume that we have a general hyperplane $S'$ containing $C'$ but not containing the cone point. There is a composition $S'\cap W\cong\bP^{2}\hookrightarrow S'\hookrightarrow\bP^{6}$ which is the Veronese embedding of $\mathbb P^2$ up to linear transformations. Considering such the embedding as 
$$
\bP^{2}\hookrightarrow\bP^{5} ; (x_{0}: x_{1}: x_{2})\mapsto(x_{0}^{2}: x_{1}^{2}: x_{2}^{2}: x_{0}x_{1}: x_{1}x_{2}: x_{0}x_{2}),
$$
with $S'\cong\bP^{5}$. One can see that $C'$ is realized as some conic in $\bP^{2}$ and its image spans $\bP^{4}$. 

Next assume all $S'$ containing $C'$ contain the cone point. In this case we will prove that there exists $S'$ containing $C'$ such that $W\cap S'$ is not integral.
There exists 1-parameter family $S'_{t}\supset C'$ of hyperplanes where $t$ is another variable. By assumption $S'_{t}$ contains the cone point. Now the intersection $W\cap S'_{t}$ is a quadratic cone and it degenerates to the union of two planes. Indeed, $W\cap S'_{t}$ is a cone over a quartic in the Veronese surface which is a conic in $\mathbb P^2$. Any $1$-parameter family of conics in $\mathbb P^2$ degenerates to the union of two lines. So the double cone breaks into the union of two irreducible surfaces, and the curve $C'$ is contained one of these and such a surface must be general because $C'$ is general.

We see that $C'$ is contained in $H' \subset W$ such that $\Phi_{|-K_X|}^*H'\in|H|$. Then we have $$-K_{H'}=-K_{W}-H'=\frac{5}{2}S'-\frac{1}{2}S'=2S'$$ and $$-K_{H'}\cdot C'-1=8-1=7$$ because of $S'\cdot C'=4$. Since the ramification is 6 points with multiplicity 2, $C'$ deforms in dimension $7-6+2=3<4$. This is a contradiction.
\end{proof}

Next we will prove Movable Bend and Break for any free $H$-conic on $X$:

\begin{theo}\label{two lines}
Let $X$ be a smooth del Pezzo threefold of Picard rank $1$ and degree $1$. Assume that $X$ is general in its moduli. 
Let $M$ be a dominant component of $\overline M_{0, 0}(X, 2)$ generically parametrizing birational stable maps. Then $C\in M$ degenerates to the union of two distinct free $H$-lines $l_{1}+l_{2}$ in $M$.
\end{theo}

\begin{proof}
Let $p\in X$ be a general point and $l$ be a general free $H$-line on $X$. Since a general $C$ is very free by Proposition~\ref{prop:veryfree}, the locus $\{C\in M\mid p\in C, l\cap C\ne\phi\}$ of $M$ is $1$ dimensional. Pick a component $N$ of this locus. We set $$\cU=\overline{\{(C, S)\in N\times|-K_{X}| \, \mid \, \text{$C'$ spans $\bP^{4}$, and  $C\subset S$}\}}\subset N\times|-K_{X}|.$$ Every fiber of the projection $p_{1}:\cU\ra N$ is a line in $|-K_{X}|$ and the image $p_{2}(\cU)$ of the other projection $p_{2}:\cU\ra|-K_{X}| = \bP^{6}$ is a surface. The locus $\cH=\{H_{1}+H_{2}\in|-K_{X}|\mid H_{i}\in|H|\}$ has dimension 4 so by looking at these dimensions, we see that we have $p_{2}(\cU)\cap\cH\ne\phi$. Then there exists a curve $C_{0}$ such that $p_{1}^{-1}(C_{0})\ni H_{1}+H_{2}$. 

If $C_{0}$ is not integral, then it is the union of two $H$-lines so $C_{0}=l_{1}+l_{2}$ and we may assume that they satisfy that $p\in l_{1}, l\cap l_{2}\ne\phi$ and $l_{1}\cap l_{2}\ne\phi$. Note that there are only finitely many such $l_{1}$ and $l_{2}$ satisfying these conditions. Indeed, there are only finitely many $l_1$ containing $p$. Any locus parametrizing $l_2$ such that $l\cap l_2 \neq \emptyset$ is $1$-dimensional, and among them there are only finitely many $l_2$ meeting with $l_1$. Since $l_1$ contains a general point it is free. Then since $l$ and $l_1$ are general in its moduli, we conclude that $l_2$ is also general in its moduli so that it is free. Thus we conclude that $l_1 + l_2$ are distinct free $H$-lines. Thus our assertion follows. 

If $C_{0}$ is integral, we may assume it is contained in $H_{1}$. Then there are only finitely many $C_{0}$ such that $p\in C_{0}, l\cap C_{0} \neq \emptyset$ and $C_{0}\subset H_{1}$. Indeed, any locus parametrizing $C_0$ containing $p$ is $2$-dimensional. In this $2$-dimensional locus, $C_0$'s meeting with $l$ are parametrized by $1$-dimensional family. A general member in this $1$-dimensional family is not contained in $H_1$. Thus we conclude that there are only finitely many $C_0$ such that  $p\in C_{0}, l\cap C_{0} \neq \emptyset$ and $C_{0}\subset H_{1}$. Then there are only finitely many $H_{1}$ containing $C_{0}$ so that $H_{1}$ is general in $|H|$ and smooth. Thus it is a smooth del Pezzo surface of degree $1$. The image $H'_{1}=\Phi_{|-K_{X}|}(H_{1})$ is the quadric cone in $\bP^{6}$, we have $-K_{H'_{1}}=2S'|_{H_1'}$. Then we have the equivalence $C'_{0}\sim2S'|_{H'_{1}}$ on $H'_{1}$ since $S'|_{H'_{1}}\cdot C'_{0}=4$. Thus we must have
$$C_{0}+\iota(C_{0})\sim(\Phi_{|-K_{X}|})^{*}C'_{0}\sim4H=-4K_{H_{1}}.$$

Computing the intersection with $C_{0}$, we have $$C_{0}^{2}+\iota(C_{0})\cdot C_{0}=4H \cdot C_{0}=8.$$ Moreover note that $2\iota(C_{0}) \cdot C_{0}=C'_{0}\cdot 3S=12$. So we conclude that $C_{0}^{2}=2$. Let $p_{a}(C_{0})$ be the arithmetic genus of $C_{0}$, then the adjunction formula $$2p_{a}(C_{0})-2=K_{H_{1}}\cdot C_{0}+C_{0}^{2}$$ implies that $p_{a}(C_{0})=1$. 
This means that there exists a blow down $\beta:H_{1}\ra\widetilde H_{1}$ between del Pezzo surfaces of degree 1 and degree 2 such that $C_{0}\sim-\beta^{*}K_{\widetilde H_{1}}$. (See, e.g., \cite[The proof of Lemma 3.4]{BLRT21} for this claim.) Then $|-K_{\widetilde{H}_1}|$ contains the union of two $(-1)$-curves $l_1 + l_2$. Since $H_{1}$ is general, the $H$-lines $l_{1},\,l_{2}$ are general and must be free.
\end{proof}

Next we will prove that the locus parametrizing the union of distinct free $H$-lines is irreducible.
One difficulty here is that the evaluation map $M_{0, 1}(X, 1) \to X$ does not admit an irreducible general fiber. To overcome this issue we will prepare the following two lemmas:
\begin{lemm}
\label{lemm:twopoints}
Assume that $X$ is general in its moduli.
Let $M_{0, 1}(X, 1)=\mathcal R_1^{(1)}$. 
Let $N \subset \mathcal R_1^{(1)} \times_X \mathcal R_1^{(1)}$ be the closure of the locus parametrizing the union of two free $H$-lines $l_1 + l_2$ with a marked intersection point such that $l_1$ and $l_2$ meet at exactly two points. Then $N$ is irreducible.
\end{lemm}

\begin{proof}
Let $l_1 + l_2$ be the union of general two free $H$-lines meeting each other at exactly two points. Then there exists a unique $H' \in |H|$ such that $H'$ contains both $l_1$ and $l_2$. Since $H$-lines are general, $H'$ is also general proving that $H'$ is a smooth del Pezzo surface of degree $1$. Then one can find a blow down $\beta : H' \to \widetilde{H}'$ to a degree $2$ del Pezzo surface such that
\[
l_1 + l_2 \sim -\beta^*K_{\wt H'}.
\]
Since the monodromy on smooth members of $|H|$ is the maximum Weyl group of $\bold E_8$ type, we conclude that the locus $N'$ of $\mathcal R_1 \times \mathcal R_1$ parametrizing the union of general free $H$-lines meeting each other at two points is irreducible. Then a natural map $N \to N'$ is a degree $2$ covering. This implies that $N$ has at most two components.

Now $N'$ admits a finite cover $N' \to |H|$. We consider a general pencil $\ell \subset |H|$ and the base change $N'_\ell \to \ell$. Note that $N'_\ell$ is irreducible by Lefschetz theorem of the monodromy. Then a point in $N'_\ell$ corresponds to $H' \in |H|$ and a birational morphism $\beta : H' \to \wt H'$ to a del Pezzo surface of degree $2$ with a pair of $(-1)$-curves $E_1, E_2$ on $\wt H'$ such that $E_1 + E_2 \sim -K_{\wt H'}$. The anticanonical sytem $-K_{\wt H'}$ defines a degree $2$ covering $\wt H' \to \mathbb P^2$ ramified along a quartic curve $B_{\wt H'}$ and the image of $E_1 + E_2$ is a bitangent line to $B_{\wt H'}$. For the quartic curve $B_{\wt H'}$, having an inflection point of order $4$ is codimension $1$ condition, so we may assume that there exists $\wt H'\in N'_\ell$ such that $\wt H'$ is smooth and $B_{H'}$ admits an inflection point of order $4$ by generality of $X$ and $\ell$. This means that $N_\ell \to N'_\ell$ is not \'etale. On the other hand, deformation theory tells us that $N_\ell$ is smooth at $(H', \beta, E_1+ E_2, p)$ where $p$ is one of points in $E_1\cap E_2$. Indeed, this follows from the fact that $X$ and $\ell$ are general.
Thus we conclude that $N_\ell$ is irreducible, proving our claim.
\end{proof}

\begin{lemm}
\label{lemm:threepoints}
Let $K\subset \mathcal R_1^{(1)} \times_X \mathcal R_1^{(1)}$ be the closure of the locus parametrizing the union of two free $H$-lines $l_1 + l_2$ with a marked intersection point such that $l_1$ and $l_2$ meet at exactly three points. Then $K$ is irreducible.
\end{lemm}

\begin{proof}
Note that for a free $H$-line $l$, its involution $\iota(l)$ is the only $H$-line meeting with $l$ at three points. Indeed, let $l'$ be a line meeting with $l$ at three points. Let $C_1$ and $C_2$ be the images of $l$ and $l'$ via $\Phi_{|-K_X|} : X \to W$ respectively, then $C_i$ is a conic. Since $W$ is the intersection of quadrics in $\mathbb P^6$, we conclude that for any plane $P$ in $\mathbb P^6$, $P\cap W$ can contain at most one conic. Thus when $C_1$ and $C_2$ are distinct, two planes spanned by $C_1$ and $C_2$ are different as well. As a result, $C_1$ can meet $C_2$ with at most two intersection points which contradicts with our assumption that $l$ and $l'$ meet at three points. Hence we conclude that $C_1 = C_2$ and $l' = \iota(l)$. Thus if we denote the closure of the locus of $\mathcal R_1 \times \mathcal R_1$ parametrizing a pair of free $H$-lines meeting at three points by $K'$, then $K'$ is isomorphic to $\mathcal R_1$.
Then a natural map $K \to K'$ is a degree $3$ covering.
The remaining of the proof is similar to a proof of Lemma~\ref{lemm:twopoints}.
\end{proof}

Using the above lemmas, one can deduce the following proposition:

\begin{prop}
Let $X$ be a smooth del Pezzo threefold of Picard rank $1$ and degree $1$. Assume that $X$ is general in its moduli. 
Then let $R \subset \mathcal R_1^{(1)}\times_X \mathcal R_1^{(1)}$ be the union of main components generically parametrizing a gluing of two distinct free $H$-lines.
Then $R$ is irreducible.
\end{prop}

\begin{proof}
Since $\mathcal R_1$ is smooth by Theorem~\ref{var of lines}, $\mathcal R_1^{(1)}$ is also smooth. Then $$\mathrm{ev}:\mathcal R_1^{(1)}\ra X$$ is a generically finite dominant cover and let $B\subset X$ be the branch divisor on $X$. Let $l\in \mathcal R_1$ be a general free $H$-line so that $l$ meets with $B$ transversally. Then $\mathrm{ev}^{-1}(l)$ is smooth and $1$-dimensional, and contains $p^{-1}(l)$ as a component where $p:\mathcal R_1^{(1)}\ra \mathcal R_1$ is the family map.

Let $D_l$ be the union of components of $\mathrm{ev}^{-1}(l)$ other than $p^{-1}(l)$. Since $D_l$ is smooth, $D_l$ is the disjoint union of smooth irreducible curves as $D_l=D_1+\cdots+D_r$. Then $D_l$ is isomorphic to the fiber at $l$ for a morphism $$R\subset \mathcal R_1^{(1)}\times_X \mathcal R_1^{(1)}\ra \mathcal R_1^{(1)}\ra \mathcal R_1$$.

For an irreducible curve $L\subset \mathcal R_1$, we show that $L\cdot p_*(D_l) >0$. Take an another line $l'$. Since $p_*(D_l)$ is algebraically equivalent to $p_*(D_{l'})$, we only have to show that $L\cdot p_*(D_{l'})>0$. The curve $L$ gives us a 1-parameter family of $H$-lines, thus $H$-lines parametrized by $L$ sweep out a surface $T$. Since $X$ has Picard rank $1$, the intersection $T\cap l'$ is not empty and so $L\cap p_*(D_{l'})\neq \phi$. 

A general $H$-line $l'$ meeting with $l$ meets with it at one point. There are finitely many $H$-lines meeting with $l$ twice and $\iota(l)$ is the only $H$-line meeting with $l$ three times.
Assume that there is $i$ such that any $H$-line parametrized by $D_i$ meets with $l$ at one point.
Then $p_*D_i$ is disjoint from $\sum_{j \neq i} p_*D_j$.
Thus we conclude that 
$$0 < p_*D_i\cdot p_*D_l=(p_*D_i)^2$$
as well as
$$0 < \left(\sum_{j \neq i}p_*D_j \right)\cdot p_*D_l = \left(\sum_{j \neq i}p_*D_j\right)^2$$
These imply that 
$$\left(p_*D_i\right)^2>0,\, p_*D_i\cdot \left(\sum_{j \neq i}p_*D_j\right)=0,\, \left(\sum_{j \neq i}p_*D_j\right)^{2}>0,$$
but this contradicts with Hodge index theorem.

Let $N$ and $K$ be the irreducible loci defined in Lemma~\ref{lemm:twopoints} and \ref{lemm:threepoints}.
Since the union of free $H$-lines is a smooth point of $\mathcal R_1^{(1)}\times_X \mathcal R_1^{(1)}$, we conclude that there are at most two components of $R$, one containing $N$ and another containing $K$. Assume that there are exactly two components, denoted by $R_N$ and $R_K$ each containing $N$ and $K$ respectively. Let $D_N$ be the union of $D_i$'s contained in $R_N$ and $D_K$ be the union of $D_i$'s contained $R_K$. Then $p_*D_N$ and $p_*D_K$ are disjoint. Thus we conclude that 
\[
(p_*D_N)^2 > 0, \, p_*D_N\cdot p_*D_K = 0,\, (p_*D_K)^2 >0.
\]
This contradicts with Hodge index theorem again.
Thus we conclude that $R$ is irreducible.
\end{proof}

Finally we prove the irreducibility of the space of $H$-conics:

\begin{theo}\label{d=2}
Let $X$ be a smooth Fano threefold such that $\Pic(X)=\bZ H$, $-K_X=2H$ and $H^3=1$. Assume that $X$ is general in its moduli. Let $M\subset\overline M_{0, 0}(X, 2)$ be a dominant component generically parametrizing birational stable maps. Then $M$ is unique.
\end{theo}

\begin{proof}
Let $R\subset \mathcal R_1^{(1)}\times_X \mathcal R_1^{(1)}$ be the main component parametrizing the union of two distinct $H$-lines and $R' \subset \overline{M}_{0,0}(X, 2)$ be the image of $R$ via the gluing map. From Lemma \ref{two lines}, there exist two distinct free $H$-lines $l_1$ and $l_2$ such that $l_1+l_2\in M$. Then $l_1+l_2$ is a smooth point of $\overline M_{0, 0}(X, 2)$. Thus $M$ contains $R'$ and it must be unique.
\end{proof}

\section{Rational curves of higher degree}
\label{sec:higherdegree}

In this section we study the space of rational curves of higher degree and prove Theorem~\ref{theo:main}. First let us recall the following theorem:

\begin{theo}[{\cite[Theorem 7.6]{LT19}}]
Let $X$ be a smooth Fano threefold such that $\Pic(X)=\bZ H$, $-K_X=2H$ and $H^3=1$.   Assume that $W$ is a component of $\overline M_{0, 0}(X, d)$ and let $W_p$ denote the sublocus parametrizing curves through a point $p\in X$. There is a finite set of points $S\subset X$ such that:
\begin{itemize}
\item $W_p$ has the expected dimension $2d-2$ for points $p$ not in $S$;
\item $W_p$ has dimension at most $2d-1$ for points $p\in S$.
\end{itemize}
Furthermore for $p\not\in S$ the general curve parametrized by $W_p$ is irreducible.
\end{theo}

\begin{proof}
See \cite[Theorem 7.6]{LT19} for a proof which uses the arguments in \cite{CS09}.
\end{proof}

As a corollary we have the following statement.

\begin{coro}\label{non-empty}
Let $X$ be a smooth Fano threefold such that $\Pic(X)=\bZ H$, $-K_X=2H$ and $H^3=1$. For any $d\geq 1$, if $\overline M_{0, 0}(X, d)$ is non-empty then every component generically parametrizes free curves and has the expected dimension.
\end{coro}

Finally we prove Theorem~\ref{theo:main} using inductive arguments on $d$:

\begin{proof}[Proof of Theorem~\ref{theo:main}]
The following proof is taken from \cite[Theorem 7,9]{LT19}. We include it for completeness of the paper. Let $\cN_d$ be the component generically parametrizing degree $d$ covers from $\mathbb P^1$ to lines. It is clear that $\cN_d$ is irreducible. By counting dimensions, we see that multiple covers of curves cannot form a component of $\overline M_{0, 0}(X)$ unless the curves are $H$-lines.

Let $M$ be a dominant component of $\overline M_{0, 0}(X, d)$ generically parametrizing birational stable maps. We claim that for a general point $x\in X$ the fiber $\mathrm{ev}^{-1}(x)\cap M$ is irreducible. Indeed, if not, then it follows from Theorem~\ref{2para} that a general stable map parametrized by $M$ is a multiple cover of a $H$-line, a contradiction with our construction. 

We will prove our theorem by induction on $d$. When $d=2$, by Theorem \ref{d=2} the dominant component $M$ generically parametrizing birational stable maps is unique. When $d>2$, we assume the assertion for $2\leq d'<d$. By gluing free curves of lower degree, there exists a component $M\subset \overline M_{0, 0}(X, d)$ such that $M\ne\cN_d$. Then by Corollary \ref{non-empty} any general member $(C, f)$ of $M$ is a birational stable map from an irreducible curve. On the other hand, by Movable Bend and Break (Theorem~\ref{mbb}), $M$ contains a chain of free curves of $H$-degree at most 2 and, by Theorem \ref{d=2}, each component of the parameter space of $H$-conics contains a chain of free $H$-lines. We see that $M$ contains a chain $(C, f)$ of free $H$-lines of length $d$ from \cite[Lemma 5.9]{LT19}. Then $(C, f)$ is a smooth point in $\overline M_{0, 0}(X)$. Furthermore if the image of $(C, f)$ is irreducible, then $(C, f)$ is contained in $\cN_d$. This implies $M=\cN_d$, a contradiction. 

Thus we conclude that the image of $(C, f)$ is reducible. So we see that $(C, f)$ is a point on the image $\Delta_{1, d-1}$ of the main component of the fiber product $\cR^{(1)}_1\times_X\cR^{(1)}_{d-1}$ which is unique by the fact that $d-1 \geq 2$ and a general fiber of $\cR_{d-1}^{(1)} \to X$ is irreducible. We conclude that $M$ contains $\Delta_{1, d-1}$ since $(C, f)$ is smooth. This means that $M$ is unique. Thus our assertion follows.
\end{proof}

\section{Applications}
\label{sec:applications}

In this section we discuss some consequences of our results for del Pezzo threefolds of Picard rank $1$ and degree $1$. In particular, we discuss applications to Geometric Manin's conjecture and enumerativity of certain Gromov-Witten invariants for these threefolds.

\subsection{Geometric Manin's conjecture}

Geometric Manin's conjecture and its refinement have been formulated in \cite[Section 6]{LT19} and \cite[Section 7]{BLRT}. We recall its formulation from \cite{BLRT}: let $X$ be a smooth weak Fano variety with $L = -K_X$. Let $f : Y \to X$ be a generically finite morphism from a smooth projective variety such that one of the following is true:
\begin{enumerate}
\item $f: Y \to X$ satisfies $a(Y, f^*L) > a(X, L) = 1$;
\item $f: Y \to X$ is dominant and it satisfies $\kappa(a(Y, f^*L)f^*L + K_Y) >0$ and 
\[
(a(X, L), b(X, L)) \leq (a(Y, f^*L), b(Y, f^*L)),
\]
in the lexicographic order, or;
\item $f : Y \to X$ is dominant and satisfies $\kappa(a(Y, f^*L)f^*L + K_Y) =0$. Moreover we have 
\[
(a(X, L), b(X, L)) \leq (a(Y, f^*L), b(Y, f^*L)),
\]
in the lexicographic order and $f$ is face contracting.
\end{enumerate}
(See \cite[Definition 3.4]{BLRT} for the definition of face contracting.)
Let $M$ be a component of $\Rat(X)$. If there is a component $N$ of $\Rat(Y)$ such that $f : Y \to X$ induces a dominant rational map $N \dashrightarrow M$, then we say $M$ is an {\it accumulating component}. If $M$ is not an accumulating component, then we say $M$ is a {\it Manin component}. 
Roughly speaking Geometric Manin's conjecture predicts that the number of Manin components for each numerical class of higher anticanonical degree is constant. For smooth Fano threefolds this is stated in \cite[Conjecture 7.9]{BLRT}. Here is a consequence of our result:

\begin{theo}
Let $X$ be a general del Pezzo threefold of Picard rank $1$ and degree $1$.
For $d \geq 2$, $\overline{M}_{0,0}(X, d)$ contains a unique Manin component.
\end{theo}

\begin{proof}
This follows from the classification of the $a$-covers: \cite[Lemma 5.2, Theorem 5.3, and Theorem 5.4]{BLRT} combined with Theorem~\ref{theo:higher_a-inv} and Theorem~\ref{theo:main}. Indeed $\mathcal N_d$ is an accumulating component as $s : \mathcal R_1^{(1)} \to X$ satisfies the property (2) above. To see this, first note that we have $\kappa(2s^*H + K_{\mathcal R_1^{(1)}}) > 0$ because there is no adjoint rigid surface $S \subset X$ with $a(S, 2H) = a(X, 2H) =1$ by Lemma~\ref{iitaka one}. Then  we have 
\[
(a(X, 2H), b(X, 2H)) = (1, 1) = (a(\mathcal R_1^{(1)}, 2s^*H), b(\mathcal R_1^{(1)}, 2s^*H)).
\] 
We conclude that component $\mathcal R_d$ is a unique Manin component.
\end{proof}

This unique Manin component is a good component in the sense of \cite[Definition 7.2]{BLRT}. In particular the above theorem confirms \cite[Conjecture 7.9]{BLRT} for general del Pezzo threefolds of Picard rank $1$ and degree $1$. See \cite[Section 7]{BLRT} for more details.

\subsection{Enumerativity of Gromov-Witten invariants}

Let $X$ be a general del Pezzo threefold of Picard rank $1$ and degree $1$.
Let $C$ be a general rational curve of $H$-degree $d \geq 2$. Then $C$ is very free by Propostion~\ref{prop:veryfree}. Thus it follows from \cite[3.14 Theorem]{Kollar} that $C$ is smooth. On the other hand \cite[Theorem 1.4]{S} implies that the normal bundle of $C$ is given by
\[
\mathcal O_{\mathbb P^1}(d-1) \oplus \mathcal O_{\mathbb P^1}(d-1).
\]
This implies that there exist finitely many rational curves of $H$-degree $d$ passing through $d$ general points. Moreover we have the following proposition:

\begin{prop}
Any stable map $f:C \to X$ of $H$-degree $d \geq 2$ whose image contains $d$ general points is a birational stable map from $\mathbb P^1$ to a rational curve of $H$-degree $d$ on $X$.
\end{prop}

\begin{proof}
One can argue as \cite[Corollary 8.4 (1)]{LT21}.
\end{proof}

Now we consider the pointed GW-invariant $\langle [pt]^d\rangle^{X, 2d}_{0, d}$.
The component $\mathcal N_d$ does not contribute to this GW-invariant as stable maps parametrized by $\mathcal N_d$ cannot pass through $d$ general points. Moreover since $\mathcal R_d$ generically parametrizes birational stable maps, the generic stabilizer of the moduli stack along this component is trivial. Altogether these imply:
\begin{theo}
The GW-invariant $\langle [pt]^d\rangle^{X, 2d}_{0, d}$ is enumerative, i.e., it coincides with the actual number of rational curves of $H$-degree $d$ passing through $d$ general points.
\end{theo}

In particular the above discussion implies that $\langle [pt]^d\rangle^{X, 2d}_{0, d}$ is non-zero. It is computed in \cite{Gol07} that $\langle [pt]^d\rangle^{X, 2d}_{0, d}$ is $60$ when $d = 1$ and it is $1800$ when $d = 2$.

\bibliographystyle{alpha}
\bibliography{Fano3-folds}

\end{document}